\documentclass[10pt]{amsart}
\usepackage{amssymb} 

\newtheorem{thm}{Theorem}[section]
\newtheorem{cor}[thm]{Corollary}
\newtheorem{lem}[thm]{Lemma}
\newtheorem{prop}[thm]{Proposition}
\theoremstyle{definition}

\newtheorem{rem}[thm]{Remark}
\newtheorem{con}[thm]{Conjecture}
\newtheorem{qu}[thm]{Question}
\numberwithin{equation}{section}

\begin{document}
\title[finite $p$-groups]{Powerful $p$-groups have  noninner
automorphisms of order $p$ and some cohomology}%
\author[Alireza Abdollahi]{Alireza Abdollahi}%
\address{Department of Mathematics, University of Isfahan, Isfahan 81746-73441, Iran; and  School of Mathematics, Institute for Research in Fundamental Sciences (IPM), P.O.Box: 19395-5746, Tehran, Iran. }%
\email{a.abdollahi@math.ui.ac.ir}%
\thanks{This research was in part supported by a grant from IPM (No. 87200118).}
\begin{abstract}
In this paper we study the longstanding conjecture of whether
there exists a noninner automorphism of order $p$ for a finite
non-abelian  $p$-group. We prove that if $G$
is a finite non-abelian $p$-group such that $G/Z(G)$ is
powerful then $G$ has a noninner automorphism of order $p$ leaving either $\Phi(G)$ or $\Omega_1(Z(G))$ elementwise fixed.
We also recall a connection between the conjecture and a cohomological problem and  we give an alternative  proof of the latter result for odd $p$, by showing that  the Tate cohomology
$H^n(G/N,Z(N))\not=0$  for all $n\geq 0$, where $G$ is a finite
$p$-group, $p$ is odd,  $G/Z(G)$ is $p$-central (i.e., elements
of order $p$ are central) and $N\lhd G$ with $G/N$  non-cyclic.
\end{abstract}
\maketitle
\section{\bf Introduction and Results}
Let $p$ be a prime number and $G$ be a non-abelian finite
$p$-group. A longstanding conjecture asserts that  $G$ admits a
noninner automorphism of order $p$ (see also Problem 4.13 of
\cite{Kbook}). By a famous result of W. Gasch\"utz \cite{G},
noninner automorphisms of $G$ of $p$-power order exist. M.
Deaconescu and G. Silberberg \cite{DS} reduced the verification
of the conjecture to the case in which $C_G(Z(\Phi(G)))=\Phi(G)$.
H. Liebeck \cite{L} has shown that finite $p$-groups of class 2
with $p>2$ must have a noninner automorphism of order $p$ fixing
the Frattini subgroup elementwise. In \cite{A} we showed the
validity of the conjecture when $G$ is   nilpotent of class 2.
In fact we proved that for any prime number $p$, every finite
non-abelian $p$-group  $G$ of class 2 has a noninner automorphism
of order $p$ leaving either the Frattini subgroup $\Phi(G)$ or
$\Omega_1(Z(G))$ elementwise fixed.\\
In Section 2 of this paper, we give some classes of $p$-groups $G$ for which the conjecture holds. We prove the validity of the conjecture whenever
 $G/Z(G)$ is  powerful, or of coclass $1$ or  $G$ is
$2$-generated and nilpotent of class $2$ with non-cyclic center (see Theorems \ref{l11}
and \ref{power1} and Corollary \ref{cor:coclass1}, below).\\
  Therefore, on the negative side, if the conjecture had a counter-example, it would not be a $p$-group of  types above.  On the positive side, we may mention nothing expect of only stating our intuition which seem very optimistic: from a result of Mann and Lubotzky \cite{LM} one knows that  any finite $p$-group $G$ has a characteristic powerful subgroup $M$ whose index is bounded by a function of the rank of $G$ and $p$; so maybe, one can prove the validity of the conjecture by an appropriate   induction argument on the rank of $G$  and/or finding a way to lift a noninner automorphism of order $p$ of $M$ to one of $G$.

The proof of Gasch\"utz's result \cite{G}  relies   on a cohomological property of finite $p$-groups. This may suggest that the cohomological tools may be still useful to attack on the conjecture. On the other hand, by using Deaconescu and Silberberg's result and a cohomological property of regular $p$-groups proved by P. Schmid \cite{S}, the validity of the conjecture is shown for regular $p$-groups. The question of which other classes of finite $p$-groups have the same cohomological property not only has its own interest and is asked in \cite{S} but also  having proved the cohomological property like regular $p$-groups, it   may be hoped (by the following means) to prove  the conjecture. So we  are also motivated  to study the latter question in Section 3.  We  explain  the cohomological property of regular $p$-groups and its connection with the conjecture.\\

 We first recall
some definitions and results concerning Tate cohomology of groups.
Let $Q$ and $A$ be finite groups where $A$ is abelian. If $Q$ acts
on $A$ (from the right) as a group, then $A$ can be viewed as a
(right) $Q$-module. We denote by $A_Q$ the submodule $\{a\in A
\;|\; a^x=a \;\text{for all}\; x\in Q\}$  of fixed points under
$Q$. The trace map $\displaystyle a\mapsto a\sum_{x\in Q} x$ of
$A$ is written $\tau=\tau_{Q}$, and its image will be denoted by
$A^\tau$. In dealing with Tate cohomology, by dimension-shifting
it is often enough to consider the situation in dimension $0$.
Recall that $H^0(Q,A)=A_Q/A^\tau$. If $Q$ and $A$ are $p$-groups,
by a theorem of Gasch\"utz and Uchida $A$ is cohomologically
trivial provided the Tate cohomology $H^n(Q,A)=0$ for just one
integer $n\geq 0$ (cf. \cite[p. 110]{Gru}). Let $G$ be a group
and  $N$ a normal subgroup of $G$. Then $G/N$ may act on $Z(N)$
as follows: $a^{gN}=a^g$ for all $a\in Z(N)$ and $g\in G$. Thus
$Z(N)$ is a $G/N$-module via this action. The group of all crossed
homomorphisms  of $G/N$ to $Z(N)$ is denoted by
$Z^1(G/N,Z(N))$ and $B^1(G/N,Z(N))$
 is the subgroup of all principal crossed homomorphisms. \\
In \cite{S}, P. Schmid proved that if $G$ is a regular $p$-group
and $N\lhd G$ such that $G/N$ is not cyclic then the Tate
cohomology $H^n(G/N,Z(N))\not=0$ for all $n$. He then conjectured that
\begin{con}
Let $G$ be a finite non-regular $p$-group. Then $H^n\big(\frac{G}{\Phi(G)},Z(\Phi(G))\big)\not=0$ for all integer $n$.
\end{con}
The following question which naturally arises from the work of Schmid will be studied in Section 3.
\begin{qu}\label{qu:Schmid}
For which  finite $p$-groups $G$ and which normal subgroups $N$ of $G$ we have
$H^n\big(\frac{G}{N},Z(N)\big)\not=0$ for all integers $n$.
\end{qu}
 A relation between
 non-triviality of Tate cohomology $H^n(G/N,Z(N))$  and the
 existence of noninner automorphisms of order $p$ in
 $Aut(G)$ is  behind the using of  the following well-known result and its corollary.
 \begin{prop}{\rm(see e.g., \cite[Result 1.1]{S1})}\label{prop:schmid1}
 Suppose that $N$ is a normal subgroup of a group $G$. Then there is a natural isomorphism $\varphi: Z^1\big(\frac{G}{N}, Z(N)\big)\rightarrow C_{Aut(G)}(N;G/N)$ given by $g^{\varphi(f)}=g\big(gN\big)^f$ for $g\in G$, $f\in Z^1\big(\frac{G}{N}, Z(N)\big)$. The image of $B^1\big(\frac{G}{N}, Z(N)\big)$ under $\varphi$ is the group of inner automorphisms of $G$ induced by $Z(N)$.
 \end{prop}
 Here $C_{Aut(G)}(N; G/N)$ denotes all automorphisms $\alpha$ of $G$ such that $x^\alpha=x$ for all $x\in N$ and $g^{-1}g^\alpha\in N$ for all $g\in G$.
 \begin{prop}{\rm(Corollary 1.2 of \cite{S1})}\label{prop:Sch}
 Assume that $N$ is a normal subgroup of a group $G$ such that $C_G(N)=Z(N)$ and $H^1(G/N,Z(N))\not=0$. Then $C_{Aut(G)}(N; G/N)$ is not contained in $Inn(G)$
 \end{prop}
   So, for applying Proposition \ref{prop:Sch},
 we need to have a normal  subgroup $N\lhd G$ such that
 $$(1) \;\;\; H^1(G/N,Z(N))\not=0 \;\;\;\;\;\text{and}\;\;\;\; (2) \;\;\; C_G(N)=Z(N).$$ By Deaconescu and Silberberg's result
 \cite{DS},  $\Phi(G)$  satisfies the  condition (2) and so we
 should verify (1). Of course non-triviality of
 $H^1(G/N,Z(N))$ is only a sufficient condition  to have a noninner
 $p$-automorphism (not necessarily of order $p$) and it is not sufficient for our purpose. Therefore according to  Propositions \ref{prop:schmid1} and \ref{prop:Sch}, the condition  $$(3) \;\;\;\;\;\; \Omega_1\big(Z^1(\frac{G}{N},Z(N))\big) \not\subseteq B^1(\frac{G}{N},Z(N))$$
 with together conditions (1) and (2)  are  sufficient to have a non-inner automorphism of order $p$ leaving both $N$ and $G/N$ elementwise fixed.  A  condition
which implies $(3)$ is the being elementary abelian of
$Z^1(G/N,Z(N))$. This is  proved for regular $p$-groups $G$ in \cite{S} whenever $N=\Phi(G)$.\\

In Section 3, we give  classes of $p$-groups  satisfying Schmid's
cohomological conclusion requested in Question \ref{qu:Schmid}. In particular we prove
$p$-groups  of class $2$ and for $p$-groups of class $3$ whenever
$p>2$ satisfy this cohomological property (see Theorem \ref{p-coh}, below). By using this result, we give an alternative proof
for the validity of  the conjecture for  $p$-groups ($p$ odd) with a powerful central factor.\\

Throughout $p$ always denotes a prime number. For a finite group
$G$, we denote by $d(G)$, $Z(G)$, $G'$, $\Phi(G)$, $Aut(G)$ and
$Inn(G)$, the minimum number of generators, the center, the
derived subgroup, the Frattini subgroup, the automorphism group,
the inner automorphism group of $G$, respectively. If $G$ is a
$p$-group, $\Omega_1(G)$ denotes the subgroup generated by
elements of order $p$. For two groups $G$ and $H$, $Hom(G,H)$
denotes the set of group homomorphisms from $G$ to $H$. If $H$ is
abelian $Hom(G,H)$ has a group structure with pointwise
multiplication. The unexplained notation is standard and follows
that of Gorenstein \cite{Gor}.
\section{\bf Finite $p$-groups without noninner automorphism of order $p$ and the existence of noninner automorphism of order $p$ in powerful $p$-groups}
Let $G$ be a group and $A$ be a normal abelian subgroup of $G$.
Then it is easy to see that  the set $[A,x]=\{[a,x] \;|\; a\in
A\}$ is a subgroup of $A$ for any element $x\in G$.  Lemmas \ref{1}, \ref{l1}  and  Corollary \ref{cor1} may be well-known, but we could not find them  as the following forms   in the published literatures.
\begin{lem}\label{1}
Let $G$ be a finite $p$-group such that  $G$ has no noninner
automorphism of order $p$ leaving  $\Phi(G)$ elementwise fixed.
Then $\Omega_1(Z(G))\leq [Z(M),g]$ for every maximal subgroup $M$
and every $g\in G\backslash M$. In particular $\Omega_1(Z(G))\leq
G'$.
\end{lem}
\begin{proof}
By the main result of \cite{DS}, we have
$\Phi(G)=C_G(Z(\Phi(G)))$ and so in particular $Z(G)\leq \Phi(G)$.
Take a maximal subgroup $M$ of $G$ and $g\in G\backslash M$.
Suppose, for a contradiction, that there exists an element
$z\in\Omega_1(Z(G))\backslash [Z(M),g]$.  Note that $z\in M$, as
$Z(G)\leq \Phi(G)$.  It is now easy to see that the map $\alpha$
on $G$ defined by $(mg^i)^\alpha=mg^iz^i$ for all $m\in M$ and
integers $i$, is an automorphism of order $p$ leaving $M$
elementwise fixed. If $\alpha$ is inner, then there exists $h\in
G$ such that $m^h=m$ for all $m\in M$ and $z=[h,g]$. Since
$G=M\langle g\rangle$, $h=m'g^i$ for some $i\in\{0,\dots,p-1\}$.
As $h\in C_G(M)$, we have $[m',m'g^i]=1$ which implies that
$[m',g^i]=1$. On the other hand, $[m'g^i,g]=z$ and so $[m',g]=z$,
as $z$ is central. Now it follows from $[m',g^i]=1$ that
$[m',g]^i=1$ and so $z^i=1$. Thus $i=0$ as $o(z)=p$. Therefore
$h=m'\in Z(M)$ and so $z\in [Z(M),g]$, a contradiction. Therefore
$\alpha$ is noninner, again a contradiction.
 This completes the proof.
\end{proof}
\begin{lem}\label{l1}
Let $G$ be a finite $p$-group such that  $G$ has no noninner
automorphism of order $p$ leaving  $\Phi(G)$  elementwise
fixed. Then
$$\Omega_1(Z(Inn(G)))\cong\Omega_1\big(\frac{Z_2(G)}{Z(G)}\big)\cong\underbrace{\Omega_1(Z(G))\times
\cdots \times \Omega_1(Z(G))}_{d(G)-\text{times}}.$$
\end{lem}
\begin{proof}
Let  $C$ be the group of all automorphisms $\phi$ of $G$ such
that $x^{-1}x^\phi\in \Omega_1(Z(G))$ and $t^\phi=t$ for all
$x\in G$ and $t\in \Omega_1(Z(G))$. As $\Omega_1(Z(G))$ is a
characteristic subgroup of $G$, $C$ is a normal subgroup of
$Aut(G)$. Note that every such automorphism $\phi$ leaves $\Phi(G)$ elementwise fixed: for, as $\phi$ is a central automorphism of $G$, it fixes $G'$ elementwise; and since $x^{-1}x^\phi$ is a central element of order at most $p$, $1=(x^{-1}x^\phi)^p=x^{-p}(x^p)^\phi$. Thus $\phi$ fixes $G^p$ elementwise. Hence $\phi$  leaves  $\Phi(G)=G'G^p$ elementwise fixed.   Now the map $\psi$ from
$$T:=Hom\big(\frac{G}{\Omega_1(Z(G))},\Omega_1(Z(G))\big)$$ to
$C$ defined by $g^{\psi(f)}=g\big(g\Omega_1(Z(G))\big)^f$ for all
$f\in T$ and $g\in G$ is a group isomorphism.
  The abelian group $T$
 is isomorphic to $$\underbrace{\Omega_1(Z(G))\times \cdots
\times \Omega_1(Z(G))}_{d'-\text{times}},$$ where $d'$ is the rank
of $\frac{G}{\Omega_1(Z(G))G'}$. By Lemma \ref{1},
$\Omega_1(Z(G))\leq G'\leq \Phi(G)$, and from which it follows
that $d'=d(G)$. Thus $C$ is an elementary abelian $p$-group of
$Aut(G)$ and so by hypothesis, $C\leq Inn(G)$ which implies that
$C\leq \Omega_1(Z(Inn(G)))$. Now let $\rho_g$ be the inner
automorphism of $G$ induced by $g$ such that $\rho_g\in
\Omega_1(Z(Inn(G)))$. Thus $g^p\in Z(G)$ and $x^{gh}=x^{hg}$ for
all $x,h\in G$. It follows easily from these relations that
$\rho_g\in C$. Hence $C=\Omega_1(Z(Inn(G)))$. This completes the
proof.
\end{proof}
The contents of  the following result must be well known.
\begin{cor}\label{cor1}
Let $G$ be a finite $p$-group such that  $G$ has no noninner
automorphism of order $p$ leaving  $\Phi(G)$  elementwise
fixed. Then $d(\frac{Z_2(G)}{Z(G)})=d(Z(G))\cdot{d(G)}$.
\end{cor}
\begin{proof}
Note that for a finite abelian $p$-group $A$, we have
$d(A)=d(\Omega_1(A))$. Now the proof follows from Lemma \ref{l1}.
\end{proof}
A finite $p$-group $G$ of order $p^n$, $n>2$, is called of coclass
$c$ whenever $G$ is nilpotent of class $n-c$.
\begin{cor}\label{cor:coclass1}
Let $G$ be a finite non-abelian $p$-group of coclass $1$. Then $G$
has a noninner automorphism of order $p$ leaving  $\Phi(G)$
elementwise fixed.
\end{cor}
\begin{proof}
Since $G$ is of maximal class, we have that $|Z_2(G)/Z(G)|=p$ and $d(G)=2$. Now Corollary \ref{cor1} completes the result.
\end{proof}
\begin{thm}\label{coclass}
Let $G$ be a finite non-abelian $p$-group of coclass $c$. If $G$
has no noninner automorphism of order $p$ leaving
$\Phi(G)$ elementwise fixed, then $$d(Z(G))(d(G)+1)\leq c+1.$$
\end{thm}
\begin{proof}
Let $d=d(G)$, $\ell=d(Z(G))$ and $|G|=p^n$. By Lemma \ref{l1},
$|Z_2(G)/Z(G)|\geq p^{\ell d}$. Since $G$ is of coclass $c$ and
$|G/Z_{n-c-1}(G)|\geq p^2$,
$$p^\ell \cdot p^{\ell d}\cdot p^{n-c-3}\cdot p^2\leq |Z(G)||Z_2(G)/Z(G)|\cdots
|G/Z_{n-c-1}(G)|=p^n.$$ Thus $p^{\ell+d\ell+n-c-3+2}\leq p^n$ and
so $\ell(d+1)\leq c+1$. This completes the proof.
\end{proof}
A $p$-group $G$ is called $p$-central whenever $\Omega_1(G)=\Omega_1(Z(G))$.
\begin{thm}\label{power1}
Let $G$ be a non-abelian finite $p$-group
such that  $G/Z(G)$ is powerful. If either $p>2$ or $Z(G)$ is not cyclic then $Aut(G)$ contains a noninner
automorphism of order $p$ leaving   $\Phi(G)$   elementwise fixed; and if $p=2$, $Aut(G)$ contains a noninner
automorphism of order $2$ leaving   either $\Phi(G)$   or $\Omega_1(Z(G))$ elementwise fixed. There is a powerful $2$-group $G$ of class $2$ such that the only  automorphisms of order $2$ leaving $\Phi(G)$ elementwise fixed are inner automorphisms of $G$.
\end{thm}
\begin{proof}
Suppose first that  $Aut(G)$ contains no noninner
automorphism of order $p$ leaving    $\Phi(G)$   elementwise fixed. By \cite[Theorem]{DS}, we may assume that $\Phi(G)=C_G(Z(\Phi(G)))$.  If $d(G)=d$, then $d(G/Z(G))\leq d$ and
it follows from \cite[1.12 and 4.2.2]{LM} that $d(Z(G/Z(G)))\leq d$. Now
Corollary \ref{cor1} implies that $d(Z\big(\frac{G}{Z(G)}\big))\geq d(Z(G))d$. This completes the proof whenever   $Z(G)$ is  non-cyclic. Thus, from now on, we further assume that $Z(G)$ is cyclic and $d(Z(G/Z(G)))=d$. Now we show that $\overline{G}=G/Z(G)$ is $p$-central. Let $x$ be an element of
order $p$ in $\overline{G}$. If $x\not\in Z(\overline{G})$, then
$d(\langle x, Z(\overline{G})\rangle)=d+1$, a contradiction, as
$d(\langle x,Z(\overline{G}) \rangle)\leq d$ by \cite[1.12 and
4.2.2]{LM}. It follows that
$\Omega_1(\overline{G})=\Omega_1(Z(\overline{G}))$.\\
Let $H$ be the subgroup of $G$ such that $H/Z(G)=\Omega_1(Z_2(G)/Z(G))$. Clearly  $H\leq Z_2(G)$ and $H/Z(G)$ is an elementary abelian group of rank $d$ and since $G$ is non-abelian, $d\geq 2$. Also $[H,G]\leq \Omega_1(Z(G))$ and $C_G(h)$ is a maximal subgroup of $G$ for all $h\in H\setminus Z(G)$: for the map $x\mapsto [h,x]$ is a group homomorphism from $G$ onto the cyclic group $\Omega_1(Z(G))$ of  order $p$ with the kernel $C_G(h)$. \\

We now show the conclusion for $p>2$.  We need to prove the following:\\

(*) \; There exist $h\in H\setminus Z(G)$ and $x\in G\setminus C_G(h)$  such that   $h^p=1$ and $(xh)^p=x^p$.\\

 Since $d\geq 2$, there are $a,b\in H\setminus Z(G)$ such that $ab^{k},a^kb\in H\setminus Z(G)$ for all integers $k$. As $a^p,b^p\in Z(G)$ and $Z(G)$ is cyclic, $a^p=b^{ps}$ or $a^{ps}=b^p$ for some integer $s$. Assume, without loss of generality, that $a^p=b^{ps}$. Now as $[a,b^{-s}]\in Z(G)$,  $a^p\in Z(G)$ and $p$ is assumed to be odd, we may write
$$(ab^{-s})^p=a^pb^{-ps}[b^{-s},a]^{\binom{p}{2}}=[b^{-s},a^p]^{\frac{p-1}{2}}=1.$$
Thus we have proved that if $p>2$, there is $h\in H\setminus Z(G)$ such that $h^p=1$.  Let $x$ be any element of $G\setminus C_G(h)$.  Since $[h,x]\in Z(G)$ and $h^p=1$ we have  $$(xh)^{p}=x^{p}h^{p}[h,x]^{\binom{p}{2}}=x^p[h^p,x]^{\frac{(p -1)}{2}}=x^p.$$
This completes the proof of  (*).\\
Now by (*), it is easy to check that the map $\beta$ on $G$ defined by $(ux^i)^\beta=u (xh)^i$ for all $u\in C_G(h)$ and all integers $i$, defines an automorphism of order $p$ which leaves $\Phi(G)$ elementwise fixed. If $\beta$ were inner, then $\beta$ would be conjugation by some element $y\in G\setminus Z_2(G)$ with $y^p \in Z(G)$. Since $G/Z(G)$ is $p$-central, it follows that $y\in Z_2(G)$, which is impossible.  This completes the proof for the case $p>2$.\\

From now on, we  assume that $p=2$ and we suppose, for a contradiction, that $Aut(G)$ contains no noninner
automorphism of order $p$ leaving  either  $\Phi(G)$ or $\Omega_1(Z(G))$   elementwise fixed. Note that $\Omega_1(Z(G))$ has only one non-trivial element $z$. \\

 First  suppose  that $H$ is non-abelian. Then there are $a,b\in H\setminus Z(G)$ such that $z=[a,b]$. Let $K=\langle a,b\rangle$. Note that $[H,G]=\langle [a,b]\rangle=K'$.  Now by an argument similar to \cite[Remark 2.2]{A}, we have that $G=KC_G(K)$. We give it here for the reader's convenience: for any $x\in G$, we have
$[a,x]=[a,b]^s$ and $[b,x]=[a,b]^t$ for some integers $s,t$. Then
$[a,b^{-s}a^tx]=1$ and $[b,b^{-s}a^tx]=1$. Hence $b^{-s}a^tx\in
C_G(\langle a,b\rangle)$ and so $G=\langle a,b\rangle C_G(\langle
a,b\rangle)$.  Note that  $Z(K)=\langle [a,b],a^2,b^2\rangle=\Phi(K)$. Now it follows from \cite[Theorem]{DS} that $K$ has a noninner automorphism  of order $2$ leaving  $Z(K)$ elementwise fixed.  Now by \cite[Remark 2.5]{A}, we have that $G$ has a noninner automorphism  of order $2$ leaving $Z(G)$ elementwise fixed, a contradiction. \\

Hence, we may assume that $H$ is abelian. Since $Z(G)$ is cyclic and $H/Z(G)$ is elementary abelian of rank $d$, one of the following may happen:\\

(1) \; $H=\langle h_1\rangle \times \cdots \times \langle h_d\rangle \times \langle h_{d+1}\rangle$, where $h_1,\dots,h_d$ are all of order $2$ and $Z(G)=\langle h_{d+1}\rangle$.\\

(2) \; $H=\langle h_1\rangle \times \cdots \times \langle h_{d-1}\rangle \times \langle h_{d}\rangle$, where $h_1,\dots,h_{d-1}$ are all of order $2$ and $Z(G)=\langle h^2_{d}\rangle$.\\

First suppose that $C_G(h_i)\not=C_G(h_j)$ for some $i$ and $j$ such that $i,j\leq d$ in the case (1) and  $d\geq 3$ and $i,j\leq d-1$ in the case (2). Since $C_G(h_i)$ and $C_G(h_j)$ are maximal subgroups of $G$, there exist elements $x_i\in C_G(h_i)\setminus C_G(h_j)$ and $x_j \in C_G(h_j) \setminus C_G(h_i)$. Thus $[x_i,h_j]=[x_j,h_i]=z$ and $G/C_G(h_i,h_j)\cong C_2\times C_2$. Now it is easy to see that the map $\phi$ on $G$ defined by $(ux_i^\ell x_j^k)^\phi=u(x_ih_i)^\ell (x_jh_j)^k$ for all $u\in C_G(h_i,h_j)$ and all integers $\ell,k$ is an automorphism of order $2$ leaving $\Phi(G)$ elementwise fixed. If $\phi$ were inner, we would have an element $g\in G\setminus H$ with $g^2\in Z(G)$, which is impossible as $G/Z(G)$ is $2$-central.   \\

Therefore, we may further assume that in the case (1): $C_G(h_1)=\cdots=C_G(h_d)$ and in the case (2): $C_G(h_1)=\cdots=C_G(h_{d-1})$.
Now we prove that the case (1) does not happen and in the case (2) we should have $d=2$. Suppose not. Then  $M=C_G(h_1)=C_G(h_2)$ is a maximal subgroup of $G$ and $h_1,h_2,h_1h_2\in H\setminus Z(G)$ are all of order 2. Let $x$ be any element of $G\setminus M$. Then $[x,h_1]=[x,h_2]=z$ and so $[x,h_1h_2]=1$ and so  $(xh_1h_2)^2=x^2$. It follows that the map $\alpha$ on $G$ defined by $(ux^i)^\alpha=u(xh_1h_2)^i$ for all $u\in M$ and all integers $i$, is an automorphism of $G$ of order $2$ leaving $\Phi(G)$ elementwise fixed. If $\alpha$ were inner, we would have an element $g\in G\setminus H$ with $g^2\in Z(G)$, a contradiction. \\

Thus it remains to settle the case (2) for $d=2$; i.e., $H=\langle h_1\rangle \times \langle h_2\rangle$, where $h_1^2=1$ and $Z(G)=\langle h_2^2\rangle$.  Now we prove that $G$ is also powerful. For this, it is sufficient to show $Z(G)\leq G^4$, since $G'Z(G)\leq G^4 Z(G)$ as $G/Z(G)$ is powerful. Note that since $[h_2,x^2]=1$ for all $x\in G$, we have $h_2\in C_G(\Phi(G))=Z(\Phi(G))$. Hence $h_2Z(G)\in \Phi(G)/Z(G)=\{a^2 Z(G) \;|\; a\in G\}$, since $G/Z(G)$ is powerful. Thus $h_2=a^2 h_2^{2s}$ for some integer $s$ and so $h_2=a^{2k}$ for some integer $k$. Therefore $Z(G)=\langle h_2^2\rangle=\langle a^{4k}\rangle\leq G^4$. Hence $G$ is a non-abelian powerful $2$-group of rank $2$.
Hence $G$ is a non-abelian ordinary metacyclic $2$-group \cite{LM} and so it follows from (\cite{King} or \cite{Hempel}) that $G$ has a presentation as following:
$$\langle a,b \;|\; a^{2^r}=b^{2^s}, b^{2^{s+t}}=1, b^a=b^{2^u+1}\rangle,$$
for some  integers $r\geq s\geq u\geq 2$ and $u\geq t\geq 0$. It follows from \cite[Corollary 4.5 and Lemma 2.1]{Hempel} that $Z(G)=\langle a^{2^{s+t-u}},b^{2^{s+t-u}}\rangle$, $G/G'\cong C_{2^r}\times C_{2^u}$, $|G|=2^{r+s+t}$, $exp(G)=2^{r+t}=o(a)$ and $o(b)=2^{s+t}$. Since $Z(G)$ is cyclic and $o(a)\ge o(b)$, it follows that $b^{2^{s+t-u}}\in \langle a\rangle \cap \langle b\rangle$. If $u>t$, then $|\langle a\rangle \cap \langle b\rangle|\geq 2^{t+1}$ and so $$2^{r+s+t}=|G|=\frac{|\langle a\rangle||\langle b\rangle|}{|\langle a\rangle \cap \langle b\rangle|}\leq 2^{r+s+t-1},$$
a contradiction. Thus $u=t\geq 2$ and $G$ has the following presentation:
$$\langle a,b \;|\; a^{2^r}=b^{2^s}, b^{2^{s+t}}=1, b^a=b^{2^t+1}\rangle, \;\;\;\text{for some integers}\;\; r\geq s\geq t\geq 2.$$
If $t=s$, then $G'\leq Z(G)$ and it follows from \cite{A} that  $G$ has a non-inner automorphism of order $2$ leaving $\Omega_1(Z(G))$ elementwise fixed.\\

 Now suppose that $s>t$. Thus $b^{2^{s-1}}\in G'$ and so  $h=b^{2^{s-1}}a^{-2^{r-1}}\not\in G'$ since $exp(G/G')=2^r$, and by $a^{2^r}=b^{2^s}$ we have that $h=b^{2^{s-1}}a^{-2^{r-1}}$ is of order $2$ so that  $h\in H=\langle g\in Z_2(G) \;|\; g^2\in Z(G)\rangle$.\\

If $r>s$, then we further have  $[b,h]=1$. Now it is easy to check (by using the latter presentation of $G$) that the map $\alpha$ on $G$ defined by $(a^ib^j)^{\alpha}=(ah)^ib^j$ for all integers $i,j$, is a non-inner automorphism of order $2$ leaving elementwise fixed $\langle b\rangle\geq \Omega_1(Z(G))$.\\

If $r=s$, then $[a,h]=[b,h]$ is of order $2$. It is now easy to check that  the map $\delta$ on $G$ defined by $(a^ib^j)^{\delta}=(ah)^i(bh)^j$ for all integers $i,j$, is a non-inner automorphism of order $2$ leaving elementwise fixed $\Omega_1(Z(G))$. This completes the proof for the case $p=2$.\\

H. Liebeck \cite[p. 272, Example]{L} considered the group $G$ with the following presentation
$$\langle a,b \;|\; a^4=[a,b,a]=1, b^8=[a,b]\rangle.$$
The group $G$ is of class $2$ and  since $G'=\langle [a,b]\rangle\leq\langle b\rangle^8\leq G^4$, $G$ is a powerful $2$-group.  As Liebeck observed the only automorphisms of
$G$ leaving $\Phi(G)$ elementwise fixed and having order 2 are of the form $\sigma$,
where $a^\sigma=av^{2r}$, $b^\sigma=bv^{2s}$ ($r = 0, 1$; $s = 0, 1$), $v=[a,b]$. These are
all inner automorphisms. The group constructed by Liebeck is of order 128. The following is a powerful $2$-group of order 64 and class 2 such that every automorphism of order $2$ leaving the Frattini subgroup elementwise fixed is an inner automorphism.
$$\langle a,b\;|\; a^4=b^4,b^{16}=1,b^4=[b,a]\rangle.$$
We leave the proof of the latter assertion to the reader.
\end{proof}
\begin{rem}
Regarding the proof of Theorem 2.6, case $p = 2$, one of the referees, who I am really  grateful to him/her for his/her ideas, has given the following argument to clarify some points of the proof:\\
The subgroup $H$ is the inverse image in $G$ of $\Omega_1(Z_2(G)/Z(G))$. At this point in the
proof, we have shown that $Z(G)$ is cyclic and that, letting $z \in Z(G)$ of order $2$, $H/Z(G)$ is isomorphic with $Hom(G/\Phi(G),\, \langle z \rangle)$; if $G/ \Phi(G)$ is
regarded as a $GF(p)$-vector space $V$, then $H/Z(G)$ is isomorphic to $V^*$ (the dual space of $V$).
Having shown that $H$ is abelian, one might let $\Omega_1(H) =\langle z \rangle \times D$, to
see that in our case (1), $C_G(D) = \Phi(G)$,
and in the case (2), $|C_G(D) : \Phi(G) | = 2$ and that, indeed, $C_G(h_i) \not= C_G(h_j)$ whenever
$1 \leq i \leq j \leq d$.
Hence (arguing as we did), case (1) is out and, in the case (2), we know that $d = 2$.
Having shown that $G$ must be itself powerful, hence ordinary metacyclic, we now invest some effort
into proving the existence of an involution contained in $Z_2(G) \setminus Z(G)$; this seems to be a
bit roundabout, since we already have got such an involution in the shape of $h_1$. 
One might argue like the following (this is much the same thing we did only shorter, and, in some sense, more transparent):\\
As  $G$ is known to be ordinary metacyclic, so $G$ has a normal
subgroup
$\langle b\rangle$ such that $G =\langle b \rangle \langle a \rangle$ for some $a\in G$ and $G' = \langle [a,\, b] \rangle
\leq \langle b^4\rangle$; in particular, $o(\langle b \rangle G') \geq 4$.
Furthermore, $z$ is the only involution in $G'$, so $h_1 \notin G'$.
If  $G/G' = \langle c_1 G'\rangle \times \langle c_2 G'\rangle$ with
$o(c_1 G') \geq 4 \leq o(c_2 G')$, then, letting $[c_i,h_1] = z^{\epsilon_i}$,
$i =1,2, \epsilon_i \in \{0,\, 1\}$, there is $\alpha \in Aut(G)$ of order 2 given by
$c_i^{\alpha} = c_i h_1^{\epsilon_i}$ , $i = 1,\,2$. As $h_1 \notin G', \alpha$ is not inner.
The only remaining possibility is that, say, $o(c_2 G')=2$. Now $b \notin \Phi(G)$ and
$o(bG') \geq 4$, so, without loss of generality, $c_1 = b$ and $c_2 = a$.
In particular, $a^2 \in G' \leq \langle b^4 \rangle$, and, as $G$ is ordinary metacyclic,
we may take $o(a) = 2$; i.e. it turns out that, in this case, actually $a = h_1$
(and, of course, $cl(G) = 2$). One might use \cite{A}, or else
point out that the automorphism $\alpha$ mapping $b$ to $ba$ ($= bh_1$) and $a$ to $a$, will do.
\end{rem}

\begin{lem}\label{z<4}
Let $G$ be a finite nilpotent $2$-generated group of class $2$.
Then $d(Z(G))\leq 3$.
\end{lem}
\begin{proof}
Suppose that $G=\langle a,b\rangle$ and let $x\in Z(G)$. Since $G$
is nilpotent of class 2, $x=a^ib^j[a,b]^k$ for some integers
$i,j,k$.  As $G'\leq Z(G)$, $[a^ib^j,a]=[a^ib^j,b]=1$. It follows
that $[a,b]^i=[a,b]^j=1$ and so $|[a,b]|=k$ divides both $i$ and
$j$. This implies that $Z(G)=\langle a^k,b^k,[a,b]\rangle$ and so
$d(Z(G))\leq 3$.
\end{proof}
\begin{thm}\label{l11}
Let $G$ be a finite $p$-group of class $3$.  If  $G/Z(G)$ is
$2$-generated and $Z(G)$ is not cyclic, then $Aut(G)$ contains a
noninner automorphism of order $p$ leaving  $\Phi(G)$  elementwise
fixed.
\end{thm}
\begin{proof}
Suppose, for a contradiction that, the conclusion is false. Then
by Corollary \ref{cor1}, the minimum number of generators of
$Z:=Z(\frac{G}{Z(G)})$ is  $4$. On the other hand, by Lemma
\ref{z<4}, $d(Z)\leq 3$, a contradiction. This completes the
proof.
\end{proof}
\section{\bf Cohomologically trivial modules and noninner $p$-automorphisms of finite $p$-groups}
In this section we prove non-triviality of Tate cohomology
$$H^n(G/N,Z(N))  \;\; \text{for all}\; n,$$ for certain $p$-groups $G$ and normal
subgroups $N\lhd G$.\\
As we mentioned in Section 1, in \cite{S1}, Schmid proved that if $G$ is a regular $p$-group
and $N$ a nontrivial normal subgroup of $G$, then
$H^n(Q,Z(N))\not=0$ for all $n$ and all non-cyclic $Q=G/N$. Here
we prove the same cohomological property for certain classes of
groups. We first need the following result concerning
cohomologically trivial groups.
\begin{lem}\label{p-central}
Let $p$ be an odd prime and $G$ be a finite  $p$-group such that
$G/Z(G)$ is $p$-central. Suppose that $A$ is a normal abelian
subgroup of $G$, $a\in A$ and $g\in G$ such that $g^p\in C_G(A)$.
If $p>3$ then $a^{g^{p-1}+\cdots+g+1}= a^p$ and if $p=3$ then
$a^{g^2+g+1}=a^3 z$ for some central element $z$ of $G$.
\end{lem}
\begin{proof}
Let $H=\langle A,g\rangle$. We have $[a,g^p]=1$  and so $(g^a
Z(G))^p=(gZ(G))^p$ for all $a\in A$. Now as $G/Z(G)$ is
$p$-central, we have $[g^a,g]\in Z(G)$ (see e.g.
\cite[Theorem 5]{LoM}). Now it follows from the Hall-Petrescu formula that
$$(g^{-1}g^a)^p=g^{-p}(g^p)^a[g^a,g^{-1}]^{\binom{p}{2}}
[g^a,g^{-1},g^{-1}]^{\binom{p}{3}}\cdots [g^a,_{p-1}g^{-1}].$$
Since $p\geq 3$, $g^p=(g^p)^a$ and $[g^a,g^{-1}]\in Z(G)$  , we
can write
$$(g^{-1}g^a)^p=([g^a,g^{-1}]^{p})^{(p-1)/2}=[(g^p)^a,g^{-1}]^{(p-1)/2}=[g^p,g^{-1}]^{(p-1)/2}=1.$$
Therefore $[g,a]^p=1$ for all $a\in A$ and so the derived
subgroup $H'$ of $H$ is of exponent dividing $p$. Now  by
the Hall-Petrescu formula we have
$$a^{g^{p-1}+\cdots+g+1}=a^p[a,g]^{\binom{p}{2}}[a,_2
g]^{\binom{p}{3}}\cdots [a,_{p-1}g].$$ Since $exp(H')$ divides
$p$, $[a,_{i}g]^{\binom{p}{i+1}}=1$ for all $i\in\{1,\dots,p-2\}$
and as $[g^a,g]\in Z(G)$, we have $[a,_2 g]\in Z(G)$ and if $p>3$,
$[a,_{p-1}g]=1$. This completes the proof.
\end{proof}
\begin{lem}\label{nil=3}
Let $p$ be an odd prime and $G$ be a finite  $p$-group of class
at most $3$.  Suppose that $A$ is a normal abelian subgroup of
$G$, $a\in A$ and $g\in G$ such that $g^p\in C_G(A)$. Then
$a^{g^{p-1}+\cdots+g+1}= a^pz$  for some central element $z$ of
$G$.
\end{lem}
\begin{proof}
Let $H=\langle A,g\rangle$. We have $[a,g^p]=1$   for all $a\in
A$.  Now by the Hall-Petrescu formula  we have
$$(g^{-1}g^a)^p=g^{-p}(g^p)^a[g^a,g^{-1}]^{\binom{p}{2}}
[g^a,g^{-1},g^{-1}]^{\binom{p}{3}}\cdots [g^a,_{p-1}g^{-1}].$$
Since  $g^p=(g^p)^a$, $p\geq 3$ and $[g^a,g^{-1}]\in Z(G)$, we
can write
$$(g^{-1}g^a)^p=([g^a,g^{-1}]^{p})^{(p-1)/2}=[(g^p)^a,g^{-1}]^{(p-1)/2}=[g^p,g^{-1}]^{(p-1)/2}=1.$$
Therefore $[g,a]^p=1$ for all $a\in A$ and so the derived
subgroup $H'$ of $H$ is of exponent dividing $p$. Now  by the
Hall-Petrescu formula  we have
$$a^{g^{p-1}+\cdots+g+1}=a^p[a,g]^{\binom{p}{2}}[a,_2
g]^{\binom{p}{3}}\cdots [a,_{p-1}g].$$ Since $exp(H')$ divides
$p$, $[a,_{i}g]^{\binom{p}{i+1}}=1$ for all $i\in\{1,\dots,p-2\}$
and as $[g^a,g]\in Z(G)$, we have $[a,_2 g]\in Z(G)$. This
completes the proof.
\end{proof}
\begin{lem}\label{p=2nil=3}
Let  $G$ be a finite  $2$-group of class at most $3$.  Suppose
that $A$ is a normal abelian subgroup of $G$, $a\in A$ and $x,y\in
G$ such that $x^2,y^2,(xy)^2\in C_G(A)$. Then $a^{xy+y+x+1}= a^4z$
for some central element $z$ of $G$.
\end{lem}
\begin{proof}
We may write
\begin{align*}
a^{xy+x+y+1}=&\\ &a^4[a,xy][a,x][a,y]=\\
&a^4[a,y][a,x]^y[a,x][a,y]=\\ &a^4 [a,x]^2[a,y]^2[a,x,y].
\end{align*}
On the other hand, for an element $g\in G$ such that $g^2\in
C_G(A)$, we have
$$1=[a,g^2]=[a,g]^g[a,g]=[a,g]^2[a,g,g],$$
and so $[a,g]^2=[a,g,g]^{-1}\in Z(G)$ as $G$ is of class $3$.
Hence, it follows from $(*)$ that $a^{xy+x+y+1}=a^4 z$ for the
central element $z=[a,x]^2[a,y]^2[a,x,y]$. This completes the
proof.
\end{proof}
\begin{lem}\label{nil=2}
Let $G$ be a finite $p$-group of class at most $2$. Then for all
$x,y\in G$ and $n\in \mathbb{N}$, there exists $z\in Z(G)$ such
that  $y^{x^{n-1}+\cdots+x+1}=y^n z$.
\end{lem}
\begin{proof}
It is straightforward as we have the identity
$$(xy)^n=x^ny^n[y,x]^{n(n-1)/2}$$ in a nilpotent group of class at
most $2$.
\end{proof}
\begin{prop}\label{coh-tr}{\rm (Proposition 1 of \cite{S})}
Suppose that $A\not=0$ is a cohomologically trivial $Q$-module
where $A$ and $Q$ are finite $p$-groups. Then for every subgroup
$H$ of $Q$, the centralizer $C_Q(A_H)=H$.
\end{prop}
\begin{thm}\label{p-coh}
Let $G$ be a finite $p$-group and $N$ be a non-trivial normal
subgroup of $G$ such that $G/N$ is not cyclic. Suppose that one
of the following holds:
 \begin{enumerate} \item
$p>2$ and $G/Z(G)$ is either nilpotent of class at most $2$ or
$p$-central;  \item  $G$ is nilpotent of class at most $2$.
\end{enumerate} Then, in any case we have
$H^n(\frac{G}{N},Z(N))\not=0$ for all $n\geq 0$.
\end{thm}
\begin{proof}
The proof follows from the lines of the proof of the Theorem in
\cite{S} but instead of using \cite[Proposition 2]{S} in the
proof, one may use Lemmas \ref{p-central}, \ref{nil=3},
\ref{nil=2}. We give the proof for the reader's convenience.\\

Suppose, for a contradiction, that $H^n(G/N,Z(N))=0$ for some
$n\geq 0$. Let $H/N$ be a subgroup of $G/N$ of order $p$. By
Gasch\"utz and Uchida's result, we have $H^0(H/N,A)=0$ where $A=Z(N)$.
Thus $A_{\frac{H}{N}}=A^{\tau_{\frac{H}{N}}}$. Now it follows from
Lemmas \ref{p-central}, \ref{nil=3} or \ref{nil=2} that there
exist elements  $z_a\in Z(G)$ ($a\in A$)  such that
$A^{\tau_{\frac{H}{N}}}=\{a^pz_a \;|\; a\in A\}$. Now since
$z_a\in Z(G)$, we have that
$$C_{\frac{G}{N}}(A^{\tau_{\frac{H}{N}}})=C_{\frac{G}{N}}(A^p).$$
Thus by Proposition \ref{coh-tr} we have that
$\frac{H}{N}=C_{\frac{G}{N}}(A^p)$. As the right hand side of the
latter equality is independent from the choice of $H/N$, we have
that $G/N$ has a unique subgroup of order $p$. Therefore $G/N$
is  cyclic  or  generalized quaternion and so in the case (1), we
are done. Thus we are left with the case (2) and we may assume
further that $G/N$ is a generalized quaternion group. In this case
as $G$ is nilpotent of class at most $2$, we have that $G/N$ is
the  quaternion group of order $8$. Now it follows from
Lemma \ref{nil=2} and a similar argument as above  that $G/N$ has
only one cyclic subgroup of order $4$, a contradiction. This
completes the proof.
\end{proof}
Let us finish by an alternative proof of Theorem \ref{power1} for the case $p>2$ in which we use Theorem \ref{p-coh}. Before that, we need the following lemma, however it has its own interest, one can see the extra amounts of works in respect to the more quick and straightforward proof of  the first part of Theorem \ref{power1}.
\begin{lem}\label{Z^1}
Let $p>2$ and $G$ be a finite  $p$-group such that $G/Z(G)$ is
$p$-central. Then $Z^1(\frac{G}{\Phi(G)},Z(\Phi(G)))$ is an
elementary abelian $p$-group.
\end{lem}
\begin{proof}
Let $f\in Z^1(\frac{G}{\Phi(G)},Z(\Phi(G)))$ and
$\bar{x}=x\Phi(G)$ for $x\in G$. We  have to prove that
$\bar{x}^{f^p}=1$. We have $a=\bar{x}^f\in Z(\Phi(G))$. Since
$x^p\in \Phi(G)$, $(\bar{x}^p)^f=1$ and so
$a^{x^{p-1}+\dots+x+1}=1$ which is equivalent to the equality
$(xa)^p=x^p$. Thus $[xa,x]\in Z(G)$ by \cite[Theorem 5]{LoM}.
Hence $[x,a]\in Z(G)$ and so $(xa)^p=x^pa^p[a,x]^{p(p-1)/2}$. On
the other hand $[(xa)^p,x]=[x^p,x]=1$ and as $[xa,x]\in Z(G)$, it
follows that $[xa,x]^p=1$. Thus $(xa)^p=x^pa^p$ and so
$a^{x^{p-1}+\dots+x+1}=a^p=1$. This completes the proof.
\end{proof}

\noindent{\bf Second Proof of part 1 of Theorem \ref{power1} for odd $p$.} Suppose, for a contradiction, that $G$ has no noninner automorphism
of order $p$ leaving  $\Phi(G)$. By an easy argument given in the first part of  of the proof of Theorem \ref{power1}, we have    that $G/Z(G)$ is
$p$-central. By \cite{DS} we may further assume that
$\Phi(G)=C_G(Z(\Phi(G)))$ which implies
$Z(\Phi(G))=C_G(\Phi(G))$. Now taking $N=\Phi(G)$ in Theorem
\ref{p-coh}, we find that
$H^1(\frac{G}{\Phi(G)},Z(\Phi(G)))\not=0$. By Lemma \ref{Z^1},
$Z^1\big(\frac{G}{\Phi(G)},Z(\Phi(G))\big)$ is an elementary abelian
$p$-group  and so it follows from \cite[Result 1.1 and Corollary
1.2]{S1} that $G$ has a noninner automorphism of order $p$
leaving  $\Phi(G)$ elementwise fixed, a
contradiction. This completes the proof. $\hfill\Box$\\

\noindent{\bf Acknowledgements.} The author is  grateful to the referees for their valuable comments and careful readings. The  ``cohomology free'' proof of the first part of Theorem \ref{power1} is due to one of the referees; who I am  indebted to him/her, as I  was inspired and encouraged   by his/her argument  to solve the second part of Theorem \ref{power1} for $p=2$. Some parts of this work was done during author's visit from University of Bath in Summer 2009. The author thanks the Department of Mathematical Sciences of University of Bath for its hospitality and especially he wishes  to thank Gunnar Traustason.   This research was partially supported by  the Center of Excellence for Mathematics, University of Isfahan and the author gratefully acknowledges the financial support of University of Isfahan for the sabbatical leave study in University of Bath.


\end{document}